\documentclass[11pt]{article}
\usepackage{amsmath,amsthm,fullpage}
\usepackage{amscd}
\usepackage{amsthm,mathtools} \usepackage{amssymb}
\usepackage{latexsym}
\usepackage{eufrak}
\usepackage{euscript}
\usepackage[width=18cm,height=24cm]{geometry}
\usepackage{epsfig}
\usepackage{tikz}
\usepackage{graphicx}
\usepackage{array}
\usepackage{enumerate} \usepackage{authblk}

\theoremstyle{theorem}

\newtheorem{theorem}{Theorem}[section]
\theoremstyle{corollary}
\newtheorem{corollary}{Corollary}[section]
\theoremstyle{lemma}
\newtheorem{lemma}{Lemma}[section]
\theoremstyle{definition}

\theoremstyle{proof}
\theoremstyle{remark}

\theoremstyle{example}
\newtheorem{example}{Example}[section]
\theoremstyle{observation}

\begin{document}
  \setcounter{Maxaffil}{2}
   \title{Spectral properties of $\mathcal{C}$-graphs}
   \author[ ]{\rm Santanu Mandal$^{1,2}$\thanks{santanu.vumath@gmail.com}}
   \author[ ]{\rm Ranjit Mehatari$^{1}$\thanks{ranjitmehatari@gmail.com, mehatarir@nitrkl.ac.in}}
   
   \affil[1 ]{Department of Mathematics, National Institute of Technology, Rourkela-769008, India}

\affil[2 ]{School of Computing Science and Artificial Intelligence,}
\affil[ ]{VIT Bhopal University, Bhopal-466114, India}
   \maketitle
 \begin{abstract}
Assumed to be undirected, simple, and connected are all of the graphs in this study, and adjacency matrix $A$ serves as the associated matrix. In this paper we show that it is possible to relate a creation sequence for a type of cographs (we call it $\mathcal{C}$-graphs).
Those cographs can be defined by a finite sequence of natural numbers. Using that sequence we obtain the inertia  of the cograph under consideration. An extended eigenvalue-free set from $(-1,0)$  to $\big{[}\frac{-1-\sqrt{2}}{2}, -1)\cup (-1, 0) \cup (0, \frac{-1+\sqrt{2}}{2}\alpha_{min}\big{]}$, (where $\alpha_{min}\geq1$ is the smallest integer of the creation sequence) is obtained for the cographs under consideration. Additionally, an exact formula is found for the characteristic polynomial.
\end{abstract}
\textbf{AMS Classification: } 05C50.\\
\textbf{Keywords:} Cograph, quotient matrix, eigenvalue-free interval, characteristic polynomial, inertia. 
\section{Introduction}

The concept of cographs was unfolded independently  by several researchers \cite{Corneil 1, Jung,Lerchs, Seinsche, Sumner} around $1970$. They give a simple structural decomposition with disjoint union and graph complement operations that can be described by a labelled tree which are used algorithmically to efficiently solve many hard problems, like, finding the maximum clique. 

We know that threshold graphs \cite{Aguilar1,Aguilar2,Bapat} and chain graphs \cite{Ala,Bell,Mandal} have creation sequence, mostly known as binary sequence or binary string. In this context, an experienced reader will surely recall that every $n$-vertex tree (even more, a labelled $n$-vertex tree) is generated by a unique sequence of $n-2$ numbers called the Pr\"{u}fer sequence~\cite{Pru}. And, of course, there are other graphs that are uniquely generated by a finite sequence in a similar way. The natural question arise: which graphs have a creation sequence (except tree, threshold and chain graphs)? We try to answer this problem partially and deduce that a type of cographs can be generated by a finite sequence of natural numbers. The key benefit of such creation sequence is that we may create an equitable partition. Because, an equitable partition is a convenient way to examine a graph's spectral features. Note that a threshold graph is a $\{P_4, C_4, 2K_2\}-$ free graph and a cograph is a $\{P_4\}-$ free graph. So a threshold graph likewise qualifies as a cograph. However, the cographs under consideration do not contain threshold graphs except possibly for few exceptions.\medskip

Recall that, a $P_4$-free graph is called a cograph and a cograph can be constructed from a single vertex by joining another cograph and by taking compliments. The class of all cographs is closed under the following operations:
\begin{itemize}
\item
disjoint union of graphs.
\item
complementation of a graph.
\end{itemize}
In this paper we only consider those cographs which can be constructed in the following way:

Let $\alpha_1,\alpha_2,\ldots,\alpha_m$ be natural numbers and $\alpha_1+\alpha_2+\cdots+\alpha_m=n$. We recursively define a cograph, denoted by $C(\alpha_1,\alpha_2,\ldots,\alpha_m)$, which satisfies the recurrence relation:
 $$C(\alpha_1,\alpha_2,\ldots,\alpha_i)=\overline{C(\alpha_1,\alpha_2,\ldots,\alpha_{i-1})\cup K_{\alpha_i}},$$
 for $i=2,3,\ldots,m$, where $C(\alpha_1)=\overline{K}_{\alpha_1}$.
 
 In other words, to construct  $C(\alpha_1,\alpha_2,\ldots,\alpha_m)$, we start with $K_{\alpha_1}$, then take its compliment. In the next step, we take disjoint union of $K_{\alpha_2}$ with the graph obtained in the first step and then take its complement. Proceeding in this way, finally, we take disjoint union of $K_{\alpha_m}$ with the graph obtained in the $(m-1)$-th step and then take its complement. A construction of the 10 vertex graph $C(3,2,2,3)$ is illustrated in the Figure \ref{Cgraph_fig1}.
 
  \begin{figure}[h]
\includegraphics[width=\textwidth]{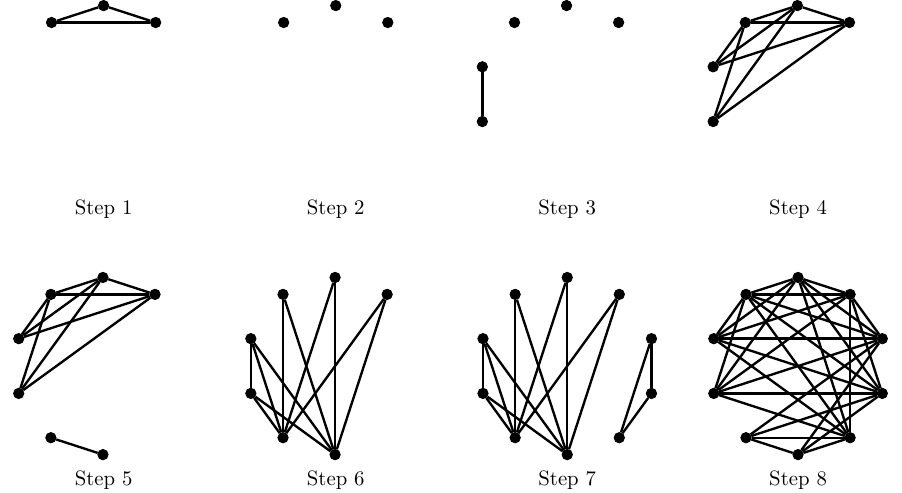}
\caption{Construction of the $\mathcal{C}$-graph $C(3,2,2,3)$}
\label{Cgraph_fig1}
\end{figure} 

  Let $\mathcal{C}$ denote the class of cographs which can be constructed in the above way, and if $G\in \mathcal{C}$ then we call $G$ a $\mathcal{C}$-graph.\medskip
  
 It may look like $\mathcal{C}$ contains a very small class of cographs, but  the class of $\mathcal{C}$-graphs is sufficiently large class of cographs. We observe that, the number of $\mathcal{C}$-graphs of order $n$ is more than the number of threshold graphs of order $n$, the number of all chain graphs of order $n$ as well as the total number of complete multipartite graphs of order $n$. In fact, there are more than 1000 $\mathcal{C}$-graphs of order 12 and more than 8388600 $\mathcal{C}$-graphs of order 25. Besides, $\mathcal{C}$ contains some of the important classes of graphs. The following graphs are categorised as $\mathcal{C}-$graphs in order to better understand how they constitute a significant large subclass of cographs.
 \begin{itemize}
 \item
  A split graph is a graph in which the vertices can be partitioned into a clique and an independent set; equivalently $\{C_4,C_5,2K_2\}$-free graph is a split graph. A complete split graph (studied in \cite{KCD1,Han1,Song}) is a split graph in which each vertex of the clique is adjacent to each vertex of the independent set. We observe that every complete split graph is a $\mathcal{C}$-graph, and it can be represented by $C(\alpha_1,\alpha_2)$. However, not all split is a $\mathcal{C}$-graph except $C(1,1,1,\cdots,1)$ and $C(1,2,1,\cdots,1)$.
 \item
 A connected graph in which atmost two vertex can have the same degree is called an antiregular graph. An antiregular graph is a $\mathcal{C}$-graph, and it's representation is either $C(1,1,1,\cdots,1)$ or $C(1,2,1,\cdots,1)$.
 \item A chordal graph is a graph with no induced subgraph isomorphic to the cycle of length four or above. Therefore a cograph is a chordal graph if it is $C_4$-free. Note that a $\{P_4,C_4\}$-free graph is called a quasi-threshold graph. Hence the $\mathcal{C}$-graph $C(\alpha_1,\alpha_2,\ldots,\alpha_{2k})$ represents a chordal graph (i.e., a quasi-threshold graph) if $\alpha_{2i}>1$ for at most one $i$, where $1\leq i\leq k$.
\item Further, the complete graph $K_n$ can be represented by $C(n-1,1)$, where as the $\mathcal{C}$-graph $C(p-1,1,q)$ is the complete bipartite graph $K_{p,q}$.
 \end{itemize}

 Throughout the paper we consider the $m$ to be even, i.e, $m=2k$ for some $k\in \mathbb{N}$, i.e., we consider those $\mathcal{C}$-graphs which can be constructed in even number of steps. \medskip

We observe that, in general a threshold graph $G$ is not a $\mathcal{C}$-graph, and if $G\in\mathcal{C}$ then $G$ must be the antiregular graph  $C(1,1,1,\cdots,1)$ or $C(1,2,1,\cdots,1)$. \medskip

For an undirected, simple graph $G$ with $n$ vertices, its adjacency matrix (generally denoted by $A$) is a square matrix whose $(i,j)$-th entry is $1$ if the vertices $i$ and $j$ are connected by an edge and $0$ otherwise.
We denote the $n$-vertex antiregular graph by $H_n$ and its adjacency matrix by $A_n$. \medskip

We now show that if $G\in \mathcal{C}$ is constructed in even number of steps, then its representation is unique.
  \begin{theorem}
 Let $G$ be a $\mathcal{C}$-graph with n vertices. If $G\in\mathcal{C}$ is generated in even number of steps, then there exists an unique integer partition $\{\alpha_1,\alpha_2,\ldots,\alpha_{2k}\}$ of $n$ such that $G=C(\alpha_1,\alpha_2,\ldots,\alpha_{2k})$. In other wards, the representation of $G$ is unique.
 \end{theorem}
 \begin{proof}
 First we observe that the largest possible order of an induced antiregular subgraph of $C(\alpha_1,\alpha_2,\ldots,\alpha_{2k})$ is  $2k$ or $2k+1$.
Therefore, if $G=C(\alpha_1,\alpha_2,\ldots,\alpha_{2k})=C(\beta_1,\beta_2,\ldots,\beta_{2l})$, then we must have $k=l$. Now we show that $\alpha_i=\beta_i$ for $i=1,2,\ldots,2k$.

Let $f:C(\alpha_1,\alpha_2,\ldots,\alpha_{2k})\rightarrow C(\beta_1,\beta_2,\ldots,\beta_{2k})$ be an isomorphism. Let $H$ be the set of vertices which are added in the final step in the construction of $C(\alpha_1,\alpha_2,\ldots,\alpha_{2k})$. Then $H$ is an independent subset of $C(\alpha_1,\alpha_2,\ldots,\alpha_{2k})$ and $d(u)=n-\alpha_{2k}$ for all $u\in H$. Then $f(H)$ must be an independent set and every $u\notin f(H)$ is adjacent to every vertex in $f(H)$. Which implies that $\alpha_{2k}=\beta_{2k}$. Then $C(\alpha_1,\alpha_2,\ldots,\alpha_{2k})\setminus H$ is isomorphic to $C(\beta_1,\beta_2,\ldots,\beta_{2k})\setminus f(H)$. Then applying  similar argument to their complement we obtain $\alpha_{2k-1}=\beta_{2k-1}$. Proceeding in this way we conclude that $\alpha_i=\beta_i$ for all $i=1,2,\ldots,k$. This completes the proof of the theorem.
 \end{proof}

%In \cite{Paulusma}, the author give  a characterization of cographs. The most significant characterization of cograph is that it is a graph with no induced subgraph isomorphic to $P_4$ \cite{Corneil 1, Corneil 2}.Brandstdt \cite{Brandstdt} studied the class of $P_4$ free graphs. Sander \cite{Sander} pointed out various applications of cograph such as biology \cite{Gagneur}, parallel computing \cite{Nakano} and so on. Cographs gain lot of attention in the last two decades. Few notable works are listed here.
In this paper we consider the adjacency matrix of a subclass of cographs, namely $\mathcal{C}$-graphs. Spectral characterization of cographs has been done previously by many authors \cite{Tura 3, Tura 4, Stanic, Chang, Ghorbani 1, He, Jacobs 2,Trevisan 1, Royle, Sander}. We mention some notable works here. Royle \cite{Royle} proved that the rank of a cograph is equal to the number of distinct non-zero rows of $A$. In \cite{Stanic}, B\i y\i ko\u{g}lu et al.~proved that $0,-1$ belong to the spectrum of a connected cograph if and only if it contains duplicate (resp. coduplicate) vertices. In \cite{Trevisan 1}, Mohhamadian and Trevisan investigated an eigenvalue-free interval for cographs. They showed that $(-1,~0)$ is an eigenvalue-free interval for the adjacency matrix of a cograph; here we extend this interval to a bigger eigenvalue exclusion set. Jacobs \emph{et~al.}~\cite{Jacobs 2} obtained a formula for the inertia of a cograph and listed a pair of equienergetic cographs. Allem and Tura \cite{Tura 3,Tura 4} wrote several papers on spectral properties of cographs. They investigated the multiplicity of eigenvalues $\lambda\neq 0,-1$. They obtained a family of non-cospectral and border-energetic cographs. They also constructed a class of integral cographs. Finally, in \cite{He}, He \emph{et al.}~established the validity of a conjecture related to the signature of chordal graphs and cographs. Allem \emph{et~al.} \cite{Tura 5} develop algorithms to obtain eigenvalue-free (within a specific range) threshold graphs (a subclass of cographs). Chen and Tura \cite{CT} also addressed this direction in a recent study using a different approach. We examined the Laplacian spectrum for this class of cographs in our earlier study \cite{MMZ}. 

\medskip
In this paper, our approach to the cographs is different. We only consider $\mathcal{C}$-graphs coded by a finite sequence. The sketch of the paper is as follows: in Section $2$, we determine the inertia and multiplicity of the trivial eigenvalues $-1,~0$ of $A$. In Section $3$, we obtain an eigenvalue-free interval for $\mathcal{C}$-graphs; in fact we show that any eigenvalue other than $0$ and $-1$ is not 
contained  in the interval $\big{[}\frac{-1-\sqrt{2}}{2}, \frac{-1+\sqrt{2}}{2}\alpha_{min} \big{]}$, where $\alpha_{min}\geq1$ is the smallest integer of the creation sequence. Finally in Section $4$, we find the determinant of a special type of tridiagonal matrix and using that we obtain the characteristic polynomial of the adjacency matrix of a $\mathcal{C}$-graph.

%%%%%%%%%%%%%%%%%%%%%%%%%%%%%%%%%%%%%%%%%%%%%%%%%%%%%%%%
%%%%%%%%%%%%%%%%%%%%%%%%%%%%%%%%%%%%%%%%%%%%%%%%%%%%%%%%

\section{Eigenvalues of $\mathcal{C}$-graphs}
Let $G=C(\alpha_1,\alpha_2,\ldots,\alpha_{2k})$ be a $\mathcal{C}$-graph on $n$ vertices. In this section, we first describe some properties of the quotient matrix  corresponding  to an equitable partition. Using these properties we determine the multiplicity of  the eigenvalues $-1$ and $0$. We also find the number of negative, zero, and positive eigenvalues of $A$.
\begin{lemma}
\label{Cgraphs_eig1_lm1}
Let $m\geq2$ and let $S=\{{i_1},{i_2},\ldots,{i_m}\}$ form a clique in a graph $G$ such that each vertex of $S$ has the same set of neighbours outside $S$. Then $-1$ is an eigenvalue of $G$ with multiplicity at least $m-1$.
\end{lemma}
\begin{proof}
For $j=1,2,\ldots,m-1$, we construct eigenvectors $X_j$ corresponding to $-1$ defined by
\begin{eqnarray*}
X_j(l)=\begin{cases}
1,&\text{if }l=i_k, k\leq j,\\
-j,&\text{if }l=i_{j+1},\\
0,&\text{otherwise.}
\end{cases}
\end{eqnarray*}
It is easy to verify that $AX_j=-X_j$ and the set $\{X_j|j=1,2,\ldots,m-1\}$ is orthogonal. So the multiplicity of $-1$ is at least $m-1$.
\end{proof}
We can use a similar argument for independent sets in a graph, and can construct an orthogonal set of eigenvectors corresponding to the eigenvalue 0.
\begin{lemma}
\label{Cgraphs_eig0_lm1}
Let $m\geq2$ and let $S=\{{i_1},{i_2},\ldots,{i_m}\}$  be an independent set in a graph $G$ such that each vertex of $S$ has the same set of neighbours outside $S$. Then the nullity of $G$ is least $m-1$.
\end{lemma}

We now recall the concept of equitable partition, which we will use to understand several spectral properties of $\mathcal{C}$-graphs. An equitable partition of a graph $G$ is a partition $\pi=\{C_1,C_2,\ldots C_m\}$ of the vertex set $V$such that, for any two parts $C_i$ and $C_j$, there is a constant $e_{ij}$ such that for all $v\in Ci$, $v$ has exactly $e_{ij}$ neighbours in $C_j$. Here the size of the equitable partition is $m$. The quotient matrix $Q_\pi$ for the equitable partition $\pi$ is a square matrix of order $m$ whose $(i,j)$-th entry is $e_{ij}$. Let $P_\pi$ be the matrix of order $n\times m$ whose  $(i,j)$-th entry is 1 if $i\in C_j$ and $0$ otherwise; the matrix $P_\pi$ is called the characteristic matrix for the equitable partition $\pi$. We now recall a result \cite{godsil} that connects these matrices.
 \begin{theorem}
Let $\pi$ be an equitable partition of a graph $G$, with quotient matrix $Q_\pi(G)$ and characteristic matrix $P_\pi$, then $AP_\pi=P_\pi Q_\pi(G)$.
 \end{theorem} 
  Now, let $\lambda$ be an eigenvalue of $Q_\pi(G)$ and let $X\in \mathbb{R}^{2k}$ be a corresponding eigenvector. Then $A(P_\pi X)=\lambda (P_\pi X)$. This implies that every eigenvalue of $Q_\pi(G)$ is also an eigenvalue of $A$. \medskip
  
Let $G$ be a $\mathcal{C}$-graph. Consider the vertex partition $\pi=\{C_1, C_2, C_3, \ldots,C_{2k}\}$ of a $\mathcal{C}$-graph $G=C(\alpha_1,\alpha_2,\ldots,\alpha_{2k})$ such that the set $C_i,\ 1\leq i\leq 2k$, contains $\alpha_i$ vertices that are added in the $i$-th step while creating $G$. Then $\pi$ is an equitable partition, and the corresponding quotient matrix $Q_\pi(G)$ is given by

  $$Q_{\pi}(G)=\begin{bmatrix}
(\alpha_1-1) &\alpha_2 &0 &\alpha_4&0 &\alpha_6& \ldots & \alpha_{2k} \\
\alpha_1&0 &0 &\alpha_4&0&\alpha_6& \ldots &\alpha_{2k}\\
0 &0&(\alpha_3-1)&\alpha_4 &0 &\alpha_6& \ldots &\alpha_{2k}\\
\alpha_1&\alpha_2&\alpha_3&0&0&\alpha_6&\ldots&\alpha_{2k}\\
0&0&0&0&(\alpha_5-1)&\alpha_6&\ldots&\alpha_{2k}\\
\alpha_1&\alpha_2&\alpha_3&\alpha_4&\alpha_5&0&\ldots&\alpha_{2k}\\
& & & & & & \ddots \\
\alpha_1&\alpha_2&\alpha_3&\alpha_4&\alpha_5&\alpha_6& \ldots & 0
\end{bmatrix}.$$

Let $D=diag[\alpha_1, \alpha_2, \alpha_3, \ldots, \alpha_{2k}]$. Then $Q_{\pi}(G)$ is similar to the symmetric matrix $D^{\frac{1}{2}}Q_{\pi}(G)D^{-\frac{1}{2}}$, thus $Q_{\pi}(G)$ is diagonalizable. Now we study the spectral properties of $Q_{\pi}(G)$.

\begin{theorem}
\label{alpha111}
Let $G=C(\alpha_1,\alpha_2,\ldots,\alpha_{2k})$ be a $\mathcal{C}$-graph of order $n$. Then $-1$ is an eigenvalue of $Q_{\pi}(G)$ if and only if $\alpha_2=1.$ Further if $\alpha_2=1$, then $-1$ is an simple eigenvalue of $Q_{\pi}(G)$.
\end{theorem}
\begin{proof}
To prove the first part of the theorem, we show that the rank of $Q_{\pi}(G)+I$ is less than $n$ if and only if $\alpha_2=1$. We have,
$$Q_{\pi}(G)+I=\begin{bmatrix}
\alpha_1 &\alpha_2 &0 &\alpha_4&0 &\alpha_6& \ldots & \alpha_{2k} \\
\alpha_1&1 &0&\alpha_4&0&\alpha_6& \ldots &\alpha_{2k}\\
0 &0&\alpha_3&\alpha_4 &0 &\alpha_6& \ldots &\alpha_{2k}\\
\alpha_1&\alpha_2&\alpha_3&1&0&\alpha_6&\ldots&\alpha_{2k}\\
0&0&0&0&\alpha_5&\alpha_6&\ldots&\alpha_{2k}\\
\alpha_1&\alpha_2&\alpha_3&\alpha_4&\alpha_5&1&\ldots&\alpha_{2k}\\
& & & & & & \ddots \\
\alpha_1&\alpha_2&\alpha_3&\alpha_4&\alpha_5&\alpha_6& \ldots & 1
\end{bmatrix}.$$
We now perform the following row operations step by step:\\
(i) $R'_{2k-i}\leftarrow R_{2k-i}-R_{2k-i-2}$ for all $i=0,1,\ldots,2k-3$.\\
(ii) $R'_2\leftarrow R_2-R_1$.\\
(iii) $R'_1\leftarrow R_1+\sum_{i=1}^{k-1} R_{2i+1}$.\\
(iv) $R_{2k}\leftrightarrow R_{2k-1}\leftrightarrow \cdots\leftrightarrow R_2\leftrightarrow R_1\leftrightarrow R_{2k}$.\medskip

After performing above operations we obtain a matrix, say $T$, which is row equivalent to $Q_\pi(G)+I$, where
$$T=\begin{bmatrix}
0&1-\alpha_2 &0 &0&\ldots &0\\
-\alpha_1 &-\alpha_2&\alpha_3-1&0 & \ldots &0\\
0&\alpha_2-1&\alpha_3&1-\alpha_4&\ldots&0\\
0&0&-\alpha_3&-\alpha_4&\ldots&0\\
& & & & \ddots \\
0 &0 &0 &0& \ldots &-\alpha_{2k}
\end{bmatrix}.$$
Here $T$ is a tridiagonal matrix with leading entry 0. So $\det T=0$ if and only if $\alpha_2=1$. Consequently, $-1$ is an eigenvalue of $T$ if and only if $\alpha_2=1$.

Next suppose that $\alpha_2=1$. Then eigenvector of $Q_\pi$ corresponding to the eigenvalue $-1$ is $X=\left[\begin{array}{ccccccc}
-1&\alpha_1&0&0&\cdots&0&0
\end{array}\right]^T$. Clearly the geometric multiplicity of the eigenvalue $-1$ is one. Thus, $Q_{\pi}(G)$ being diagonalizable, the algebraic multiplicity of the eigenvalue $-1$ is also one. Hence the result follows.
\end{proof}

In the context of the paper, the preceding result is important. For the equitable partition $\pi$ of a $\mathcal{C}$-graph $G$, each vertex set $C_i$ either forms a clique or an independent set. This makes it easier in figuring out the multiplicity and inertia of the eigenvalue $-1$. We go into further detail about this in the following section.

\subsection{Inertia of a $\mathcal{C}$-graph}
Let $B$ be a square matrix and let $n_{-}(B),~n_{0}(B)$ and $n_{+}(B)$ denote the number of negative, zero and positive eigenvalues of $B$ respectively. The triplet $(n_{-}(B),~n_{0}(B),~n_{+}(B))$ is called the inertia of $B$. The inertia of the adjacency matrix of a graph is known as the inertia of the graph. We now recall some important well-known results (for more details we refer to the book \cite{Horn} by Horn and Johnson).  

\begin{theorem}[\textbf{Weyl's inequality}]
\label{Weyl}
Let $B, ~C$ be Hermitian matrices of order $n$ and let the respective eigenvalues of $B,~C$, and $B+C$ be $\{ \lambda_i(B)\}_{i=1}^{n}$, $\{ \lambda_i(C\}_{i=1}^{n}$, $\{ \lambda_i(B+C)\}_{i=1}^{n}$, each algebraically ordered as below:
\begin{equation}
\lambda_1\leq \lambda_2 \leq \lambda_3 \leq \cdots \leq \lambda_n.
\end{equation}
Then $$\lambda_i(B)+ \lambda_1(C) \leq \lambda_i(B+C) \leq \lambda_i(B)+ \lambda_n(C),$$ where $i=1,2,\ldots,n.$
\end{theorem}

\begin{theorem}[\textbf{Sylvester's  law of inertia}]
\label{Sylvester}
Let  $B$ and $C$  be Hermitian matrices of equal order.  Then $B$ and $C$ have the same inertia if and only if $B=SCS^{*}$ for some non-singular matrix $S$, where $S^{*}$ is the conjugate transpose of $S$.
\end{theorem}

\begin{theorem}[\textbf{Ostrowski Theorem}]
\label{Ostrowski}
 Let $B,~S$ be matrices of order $n$ with $B$ Hermitian and $S$ non-singular. Let the eigenvalues of  $B,~SBS^{*}$, and $SS^{*}$ be arranged in non-decreasing order as in $(1)$. Let  $\sigma_1\geq \sigma_2 \cdots \geq \sigma_n$ be the singular values of $S$. Then for each $p=1,2, \ldots n$, there is positive real number $\theta_p \in[\sigma_n^2,~\sigma_1^2]$ such that $$\lambda_p(SBS^{*})=\theta_p \lambda_p(B) $$ 
\end{theorem}

We now compute the inertia of the $\mathcal{C}$-graph $G=C(\alpha_1,\alpha_2,\ldots,\alpha_{2k})$. 

\begin{theorem}
Let $G=C(\alpha_1,\alpha_2,\ldots,\alpha_{2k})$ be a  $\mathcal{C}$-graph. Then the inertia of the quotient matrix $Q_{\pi}(G)$ is given by
$$\big{(}n_{-}(Q_{\pi}(G)),~n_{0}(Q_{\pi}(G)),~n_{+}(Q_{\pi}(G))\big{)}=\big{(}k,~0,~k\big{)}.$$
\end{theorem}
\begin{proof}
The quotient matrix of the graph $G$ is
 $$Q_{\pi}(G)=\begin{bmatrix}
(\alpha_1-1) &\alpha_2 &0 &\alpha_4&0 &\alpha_6& \ldots & \alpha_{2k} \\
\alpha_1&0 &0 &\alpha_4&0&\alpha_6& \ldots &\alpha_{2k}\\
0 &0&(\alpha_3-1)&\alpha_4 &0 &\alpha_6& \ldots &\alpha_{2k}\\
\alpha_1&\alpha_2&\alpha_3&0&0&\alpha_6&\ldots&\alpha_{2k}\\
0&0&0&0&(\alpha_5-1)&\alpha_6&\ldots&\alpha_{2k}\\
\alpha_1&\alpha_2&\alpha_3&\alpha_4&\alpha_5&0&\ldots&\alpha_{2k}\\
& & & & & & \ddots \\
\alpha_1&\alpha_2&\alpha_3&\alpha_4&\alpha_5&\alpha_6& \ldots & 0
\end{bmatrix}.$$
Consider the diagonal matrices: $$D=diag[\alpha_1, \alpha_2, \alpha_3, \ldots, \alpha_{2k} ],\ D{'}=diag\Big[\frac{\alpha_1-1}{\alpha_1}, 0, \frac{\alpha_3-1}{\alpha_3},0 \ldots, \frac{\alpha_{2k-1}-1}{\alpha_{2k-1}}, 0\Big].$$ Then $Q_{\pi}(G)$ can be expressed as
\begin{equation}
\label{Cgraphs_eq2}
Q_{\pi}(G)=D^{\frac{1}{2}}(A_{2k}+D{'})D^{\frac{1}{2}},
\end{equation}
where $A_{2k}$ denotes the adjacency matrix of the antiregular graph $H_{2k}$.
Let $C=(A_{2k}+D{'})$. By Theorem \ref{Sylvester}, $Q_{\pi}(G)$ and $C$ both have the same inertia. 

 Let the eigenvalues of $Q_{\pi}(G),~A_{2k},~D,~D{'},$ and $C$ be arranged in non-decreasing order as in $(1)$.
It is well-known (see \cite{Bapat}) that the antiregular graph $H_{2k}$ has the inertia $\big{(}k,~0,~k\big{)}$.\\
By Theorem \ref{Weyl}, we have
$$\lambda_{k+1}(C)\geq \lambda_{k+1}(A_{2k})+\lambda_1(D{'})>0.$$
And
$$\lambda_{k}(C)\leq \lambda_{k}(A_{2k})+\lambda_{2k}(D{'})<0,$$
since $\lambda_{k}(A_{2k})\leq-1$ and all the eigenvalues of $D{'}$ lie in the interval $[0, 1)$. \\
This completes the proof.
\end{proof}
As a consequence of this theorem, we have the following corollaries.
\begin{corollary}
For the $\mathcal{C}$-graph $G=C(\alpha_1,\alpha_2,\ldots,\alpha_{2k})$, the inertia is given by
$$\big{(}n_{-}(A),~n_{0}(A),~n_{+}(A)\big{)}=\big{(}\sum_{i=1}^{k}\alpha_{2i-1},~\sum_{i=1}^{k}\alpha_{2i}-k,~k\big{)}$$
\end{corollary}

\begin{corollary}
Let $m(G)$ denote the number of distinct eigenvalue of the adjacency matrix $A$ of a $\mathcal{C}$-graph $G=C(\alpha_1,\alpha_2,\ldots,\alpha_{2k})$. Then 
$ m(G) \leq 2k+2. $
\end{corollary}

In light of Theorem \ref{alpha111} and Lemmas \ref{Cgraphs_eig1_lm1} and \ref{Cgraphs_eig0_lm1}, we emphasize the following result which counts the frequency of occurrence of the eigenvalues $0,~-1$.
\begin{corollary}
Let $m_0$ and $m_{-1}$ denote the multiplicities of $0$ and $-1$ as an eigenvalue of $A$ of a $\mathcal{C}$-graph $G=C(\alpha_1,\alpha_2,\ldots,\alpha_{2k})$. Then 
$$m_{0}=\sum_{i=1}^{k}\alpha_{2i}-k~~~\text{and}~~~m_{-1}=\begin{cases}
\sum_{i=1}^{k}\alpha_{2i-1}-k,&\text{if }\alpha_2\neq1,\\
\sum_{i=1}^{k}\alpha_{2i-1}-k+1,&\text{if }\alpha_2=1.
\end{cases}$$
\end{corollary}

\section{Eigenvalue-free interval for $\mathcal{C}$-graphs}
We now examine the important problem of determining an eigenvalue-free interval for $\mathcal{C}$-graphs. For a cograph $G$, the eigenvalues $0$ and $-1$ are called the trivial eigenvalues  and all other eigenvalues are called non-trivial eigenvalues. We are interested to find out an interval which does not contain any non-trivial eigenvalue of a $\mathcal{C}$-graph. Let $\lambda_{+}(G)$ be the smallest positive eigenvalue of $G$ and $\lambda^{-}(G)$ be the largest eigenvalue of $G$  smaller than $-1$. First we recall a result due to Aguilar et al.
\begin{lemma}{\cite{Aguilar2}}
\label{lemma1}
Let $n$ be a positive integer. Then
$$\lambda^{-}(H_n)<\frac{-1-\sqrt{2}}{2},\text{ and } \lambda_{+}(H_n)>\frac{-1+\sqrt{2}}{2}.$$

In other words, the closed interval $\big{[}\frac{-1-\sqrt{2}}{2}, \frac{-1+\sqrt{2}}{2}\big{]}$ does not contain any non-trivial eigenvalue of the antiregular graph $H_n$.
\end{lemma}

We are now in a position to establish the main result of this section.
\begin{theorem}
Consider the $\mathcal{C}$-graph $G=C(\alpha_1,\alpha_2,\ldots,\alpha_{2k})$. Then 
$$\lambda^{-}(G)<\frac{-1-\sqrt{2}}{2},\text{ and }\lambda_{+}(G)>\frac{-1+\sqrt{2}}{2}\alpha_{min},$$
where $\alpha_{min}=\displaystyle \min_{1\leq i\leq 2k}\{\alpha_i\}$.
\end{theorem}
\begin{proof} 
First observe that any nontrivial eigenvalue of $G$ must be an eigenvalue of the quotient matrix $Q_\pi(G)$.
Recall the representation of $Q_\pi(G)$ described in equation (\ref{Cgraphs_eq2}):
$$Q_{\pi}(G)=D^{\frac{1}{2}}(A_{2k}+D{'})D^{\frac{1}{2}},$$
where $$D=diag[\alpha_1, \alpha_2, \alpha_3, \ldots, \alpha_{2k} ],\text{ and } D{'}=diag\bigg{[}\frac{\alpha_1-1}{\alpha_1}, 0, \frac{\alpha_3-1}{\alpha_3},0 \ldots, \frac{\alpha_{2k-1}-1}{\alpha_{2k-1}}, 0\bigg{]}.$$
Let us now recall the expressions $(1)$ and (2). \\
By Theorem \ref{Ostrowski}, we have 
$$\lambda_{k}(Q_{\pi}(G))=\theta_{k} \lambda_k (A_{2k}+D{'}),$$
where $\theta_k \in [\alpha_{min},~\alpha_{max}]$, and $\alpha_{max}=\displaystyle \max_{1\leq i\leq 2k} \{\alpha_{i}\}$. \\

Therefore by Lemma $\ref{lemma1}$ and using Theorem \ref{Sylvester}, 

$$\lambda_{k+1}(Q_{\pi}(G))=\theta_{k+1} \lambda_{k+1} (A_{2k}+D{'})>\frac{-1+\sqrt{2}}{2}\theta_{k+1}.$$

Thus
$$ \lambda_{k+1}(Q_{\pi}(G))>\frac{-1+\sqrt{2}}{2}\alpha_{min}. $$

That is
$$\lambda_{+}(G)>\frac{-1+\sqrt{2}}{2}\alpha_{min}.$$
\medskip
 
We now prove $\lambda^{-}(G)<\frac{-1-\sqrt{2}}{2}$.\\
First note that the $\mathcal{C}$-graph $C(1,1,\ldots,1)$ and $C(1,2,1,\ldots,1)$ are antiregular graphs on $2k$ and $2k+1$ vertices respectively. So clearly $H_{2k}$ is a subgraph of $G$ and further if $\alpha_2\neq1$, then $H_{2k+1}$ is also a subgraphs of $G$. Let the eigenvalues of $A_{2k}$ and $A_{2k+1}$ be arranged in non-decreasing order as given in $(1)$. Then we have 
$$\lambda_{k}(A_{2k})=-1,\text{ and }\lambda_{k-1}(A_{2k})<\frac{-1-\sqrt{2}}{2}.$$ 
Also we have,
$$\lambda_{k+1}(A_{2k+1})=0,\text{ and } \lambda_{k}(A_{2k+1})<\frac{-1-\sqrt{2}}{2}. $$

We consider two cases:\\
\textbf{Case 1.} $\alpha_2=1$. Then $-1$ is the largest negative eigenvalue $G$, and so $\lambda^-(G)=\lambda_{k-1}(A)$. Since $H_{2k}$ is a subgraph of $G$,  by the interlacing theorem 
$$\lambda_{k-1}(A) \leq \lambda_{k-1}(A_{2k})<\frac{-1-\sqrt{2}}{2}.$$ 
\textbf{Case 2.} $\alpha_2 \neq1$. So $\lambda_{k}(A)<-1$. Since $H_{2k+1}$ is a subgraph of $G$, by the interlacing theorem
$$\lambda_{k}(A) \leq \lambda_{k}(A_{2k+1})<\frac{-1-\sqrt{2}}{2}.$$ 
Combining above two cases we conclude that $\lambda^{-}(G)<\frac{-1-\sqrt{2}}{2}$.\\
 Hence the theorem is proved.
\end{proof}

\begin{corollary}

For the $\mathcal{C}$-graph $G=C(\alpha_1,\alpha_2,\ldots,\alpha_{2k})$, the interval  $\big{[}\frac{-1-\sqrt{2}}{2}, \frac{-1+\sqrt{2}}{2}\alpha_{min} \big{]}$ does not contain any non-trivial eigenvalue of $G$.

\end{corollary}

%%%%%%%%%%%%%%%%%%%%%%%%%%%%%%%%%%%%%%%%%%%%%%%%%%%%%%%
%%%%%%%%%%%%%%%%%%%%%%%%%%%%%%%%%%%%%%%%%%%%%%%%%%%%%%%

\section{Characteristic polynomial}
This section is about the characteristic polynomial of a $\mathcal{C}$-graph $G=C(\alpha_1,\alpha_2,\ldots,\alpha_{2k})$ of order $n$. We first obtain an expression for the characteristic polynomial of the quotient matrix $Q_{\pi}(G)$, and then use it to calculate the characteristic polynomial of $A$. Suppose $\Psi(x)$ and $\Psi_\pi(x)$ respectively denote the characteristic polynomial of $A$ and $Q_\pi(G)$. Then,

$$\Psi(x)=x^{\sum_{i=1}^{k}(\alpha_{2i}-1)}(x+1)^{\sum_{i=1}^{k}(\alpha_{2i-1}-1)}\Psi_{\pi}(x).$$

%For $1\leq i\leq k$, consider the finite subsequence $\{\alpha_1,\alpha_2,\ldots,\alpha_{2i}\}$ of $\{\alpha_1,\alpha_2,\ldots,\alpha_{2k}\}$; and let $G_i=C(\alpha_1,\alpha_2,\ldots,\alpha_{2i})$. Let $\Phi_i(x)$ denote the characteristic polynomial of the quotient matrix of $G_i$. Then $\Phi_k(x)=\psi_\pi(x)$.\\
Let $\beta_1,\beta_2,\ldots,\beta_n$ be real numbers. Now  consider the following tridiagonal determinant:
$$T(\beta_1,\beta_2,\ldots,\beta_n)=\begin{vmatrix}
\beta_1&-1\\
1&\beta_2&-1\\
&1&\beta_3&-1\\
& &\ddots & \ddots& \ddots \\
&&&1&\beta_{n-1}&-1\\
&&&&1&\beta_n
\end{vmatrix}.$$
The above determinant can be found by using the three-term recurrence relation for the determinant of tridiagonal matrices. We provide an alternative way to find it. \\Let $I_n=\{1,2,\ldots,n\}$ and let $S_n$ denote the set of all permutations on $I_n$. We call a permutation $e$-transposition if it is a two  cycle of the form $(l,\ l+1)$, for some $1\leq l<n$. Let $E_n$ denote the set of all permutations which are a product of disjoint $e$-transpositions. Also we consider the identity permutation to be an element of $E_n.$

First we prove the following lemma.
\begin{lemma}
\label{Cgraphs_cp_lm1}
Let $\beta_1,\beta_2,\ldots,\beta_n$ be real numbers, then 
$$T(\beta_1,\beta_2,\ldots,\beta_n)=\sum_{\sigma\in E_{n}}\prod_{\substack{i\in I_n,\\ \sigma(i)=i}}\beta_i.$$
\end{lemma}
\begin{proof}
Let $T=[t_{ij}]$ be the matrix corresponding to the determinant $T(\beta_1,\beta_2,\ldots,\beta_n)$.
By the Leibniz formula for determinants:
$$T(\beta_1,\beta_2,\ldots,\beta_n)=\det T=\sum_{\sigma\in S_n}sgn(\sigma)\prod_{i\in I_n}t_{i,\sigma(i)}.$$
Now $t_{i,\sigma(i)}$ is (possibly) nonzero if $\sigma(i)=i-1$ or $i$ or $i+1$. Thus for a nonzero $t_{1,\sigma(1)}t_{2,\sigma(2)}\ldots t_{n,\sigma(n)}$, the permutation $\sigma$ must be a product of distinct $e$-transpositions. Now suppose $\sigma=(i_1,\ i_1+1)(i_2,\ i_2+1)\ldots (i_l,\ i_l+1)$ be a product of distinct $e$-transpositions. Then $sgn(\sigma)=(-1)^l$ and 

\begin{equation*}\begin{split}
t_{1,\sigma(1)}t_{2,\sigma(2)}\ldots t_{n,\sigma(n)}&=(-1)^l\prod_{j\notin\{i_1, i_1+1,i_2,i_2+1,\ldots, i_l, i_l+1\}}\beta_j\\
&=(-1)^l\prod_{\sigma(j)=j}\beta_j.
\end{split}
\end{equation*}
Now taking summations over all permutations of $I_n$, we get our result.
\end{proof}
\begin{theorem}
Let $G=C(\alpha_1,\alpha_2,\cdots,\alpha_{2k})$ be a $\mathcal{C}$-graph of order $n$. Then the characteristic polynomial of  the quotient matrix $Q_\pi(G)$ is given by $$\psi_\pi(x)=\prod_{i=1}^k(\alpha_{2i-1}-1-x)(\alpha_{2i}+x)T(\beta_1,\beta_2,\ldots,\beta_{2k}),$$

where $\beta_1=\dfrac{1+x}{\alpha_1-1-x}$,  and $\beta_i=\begin{cases}\frac{\alpha_i}{\alpha_{i}-1-x}&\text{ if }i \text{ odd and  }i\geq3\\
\frac{\alpha_i}{\alpha_{i}+x}&\text{ if }i \text{ even.  }
\end{cases}$ 
\end{theorem}

\begin{proof}
The characteristic polynomial of $Q_{\pi}(G)$ is
 $$\psi_\pi(x)=\det (Q_\pi(G)-xI)=\begin{vmatrix}
(\alpha_1-1-x) &\alpha_2 &0 &\alpha_4& \ldots & 0&\alpha_{2k} \\
\alpha_1&-x &0 &\alpha_4&\ldots &0&\alpha_{2k}\\
0 &0&(\alpha_3-1-x)&\alpha_4 & \ldots &0&\alpha_{2k}\\
\alpha_1&\alpha_2&\alpha_3&-x&\ldots&0&\alpha_{2k}\\
& & & & \ddots \\
0&0&0&0&\ldots&(\alpha_{2k-1}-1-x)&\alpha_{2k} \\
\alpha_1&\alpha_2&\alpha_3&\alpha_4&\ldots &\alpha_{2k-1}&-x
\end{vmatrix}.$$

We perform $R'_{2k-i}\leftarrow R_{2k-i}-R_{2k-i-2}$ for all $i=0,1,\ldots,2k-3$; and $ R_2^{'}\gets R_2-R_1$. We obtain
$$\psi_\pi(x)\equiv \begin{vmatrix}
(\alpha_1-1-x) &\alpha_2 &0 &\alpha_4& \ldots & 0&\alpha_{2k} \\
(1+x)&-(\alpha_2+x) &0 &0&\ldots &0&0\\
-(\alpha_1-1-x) &-\alpha_2&(\alpha_3-1-x)&0 & \ldots &0&0\\
0&(\alpha_2+x)&\alpha_3&-(\alpha_4+x)&\ldots&0&0\\
& & & & \ddots \\
0&0&0&0&\ldots&(\alpha_{2k-1}-1-x)&0 \\
0&0&0&0&\ldots &\alpha_{2k-1}&-(\alpha_{2k}+x)
\end{vmatrix}$$
Again performing $R_1^{'}\gets R_1+\sum_{i=1}^{k-1}R_{2i+1}$, we obtain
$$\psi_\pi(x) \equiv \begin{vmatrix}
0 &0 &0 &0& \ldots &(\alpha_{2k-1}-1-x)&\alpha_{2k} \\
(1+x)&-(\alpha_2+x) &0 &0&\ldots &0&0\\
-(\alpha_1-1-x) &-\alpha_2&(\alpha_3-1-x)&0 & \ldots &0&0\\
0&(\alpha_2+x)&\alpha_3&-(\alpha_4+x)&\ldots&0&0\\
& & & & \ddots \\
0&0&0&0&\ldots&(\alpha_{2k-1}-1-x)&0 \\
0&0&0&0&\ldots &\alpha_{2k-1}&-(\alpha_{2k}+x)
\end{vmatrix}$$

Taking permutation on rows we obtain a tridiagonal determinant.
$$\psi_\pi(x) \equiv \begin{vmatrix}
(1+x)&-(\alpha_2+x) &0 &0&\ldots &0&0\\
-(\alpha_1-1-x) &-\alpha_2&(\alpha_3-1-x)&0 & \ldots &0&0\\
0&(\alpha_2+x)&\alpha_3&-(\alpha_4+x)&\ldots&0&0\\
0&0&-(\alpha_3-1-x)&-\alpha_4&\ldots&0&0\\
& & & & \ddots \\
0&0&0&0&\ldots &\alpha_{2k-1}&-(\alpha_{2k}+x)\\
0 &0 &0 &0& \ldots &-(\alpha_{2k-1}-1-x)&-\alpha_{2k}
\end{vmatrix}$$
Now, for $1\leq i\leq k$, we multiply $2i$-th row by $-1$. We also use the fact that sign interchange between $ij$-th and $ji$-th of a tridiagonal matrix do not change the determinant. So we obtain
\begin{equation*}
\begin{split}
\psi_\pi(x)&=(-1)^k\begin{vmatrix}
(1+x)&-(\alpha_2+x) &0 &\ldots &0&0\\
(\alpha_1-1-x) &\alpha_2&-(\alpha_3-1-x) & \ldots &0&0\\
0&(\alpha_2+x)&\alpha_3&\ldots&0&0\\
0&0&(\alpha_3-1-x)&\ldots&0&0\\
& & &  \ddots \\
0&0&0&\ldots &\alpha_{2k-1}&-(\alpha_{2k}+x)\\
0 &0  &0& \ldots &(\alpha_{2k-1}-1-x)&\alpha_{2k}
\end{vmatrix}\\
&=(-1)^k(\alpha_1-1-x)(\alpha_2+x)(\alpha_3-1-x)(\alpha_4+x)\cdots(\alpha_{2k}+x)T(\beta_1,\beta_2,\ldots,\beta_{2k}).
\end{split}
\end{equation*}
The last form in the previous equation is obtained by taking $(\alpha_1-1-x)$ common from first column, $(\alpha_2+x)$ from the second column, $(\alpha_3-1-x)$ from the third column and finally $(\alpha_{2k}+x)$ from the last column.
This completes the proof of the theorem.
\end{proof}
We have a look at an example to help us comprehend the above theorem before wrapping up this work.
\begin{example}

Consider the $\mathcal{C}$-graph $G=C(4,3,2,2)$ with $11$ vertices. Then
$$E_n=\{ id, (1, 2), (2, 3), (3, 4), (1, 2)(3, 4)  \}.$$
Therefore, \\
 \begin{align*}
  T(\beta_1,\beta_2,\beta_3,\beta_4) =&\beta_1 \beta_2 \beta_3 \beta_4+\beta_1 \beta_2 +\beta_1 \beta_4 +\beta_3 \beta_4 +1 \\
 =& \frac{12(1+x)}{(3-x)(3+x)(1-x)(2+x)} + \frac{3(1+x)}{(3-x)(3+x)} + \frac{2(1+x)}{(3-x)(2+x)}\\&+\frac{4}{(1-x)(2+x)}+1.
\end{align*} 
Therefore 
\begin{align*}
\Psi_{\pi}(x)=&(-1)^2(3-x)(3+x)(1-x)(2+x)T(\beta_1,\beta_2,\beta_3,\beta_4)\\=& 12(1+x)+3(1+x)(1-x)(2+x)+2(1+x)(3+x)(1-x)+4(3+x)(3-x)\\&+(3-x)(3+x)(1-x)(2+x).
\end{align*}

Simplifying the expression, we obtain
$$\Psi_{\pi}(x)= x^4-4x^3-27x^2+8x+78.$$
Therefore,
$$\Psi(x)=x^3(x+1)^4 \big{[} x^4-4x^3-27x^2+8x+78   \big{]}.$$
\end{example}

\section{Conclusions}
In this work, the primary attention is paid to obtaining an extended eigenvalue-free interval for a subclass of cographs; we call it a $\mathcal{C}$-graph. The class of $\mathcal{C}$-graphs is a large subclass of cographs that can be represented by a finite creation sequence (of natural numbers). However, it is not clear that we can give a similar thing for cographs. Threshold graphs are another class of cographs that has a (binary) creation sequence. The question remains: Is there a creation sequence for each cograph?
\medskip

Compared to the previous method, our approach to the cographs is entirely different. We obtain an explicit formula for the characteristic polynomial by observing that an equitable partition of a $\mathcal{C}$-graph is determined by its creation sequence. For a $\mathcal{C}$-graph $G$, we show that the set $\big{[}\frac{-1-\sqrt{2}}{2}, -1)\cup (-1, 0) \cup (0, \frac{-1+\sqrt{2}}{2}\alpha_{min}\big{]}$ contains no eigenvalues of $G$, where $\alpha_{min}\geq1$ is the smallest integer of the creation sequence.

%%%%%%%%%%%%%%%%%%%%%%%%%%%%%%%%%%%%%%%%%%%%%%%%%%%%%%%%%%%%%%%%%%%%%%%%%%%%%%%%%%%%%%%%%%%%%%%%%%%%%%%%%%%%%%%%
\section{Acknowledgements}
The  research of Santanu Mandal is supported by the University Grants Commission (UGC), India under the beneficiary code BININ01569755. This author also acknowledges the infrastructure provided by the VIT Bhopal University, India. 
\section{Statements and Declarations} \textbf{Competing Interests:} On behalf of all authors, the corresponding author states that there is no conflict of interest.

\section{Data availability}
Data sharing is not applicable to this article as no data were created or analyzed in this study.


\begin{thebibliography}{99}

\bibitem{Aguilar2}
C.O. Aguilar, M. Ficarra, N. Shurman, and B. Sullivan, \textit{The role of anti-regular graph in the spectral analysis of threshols graphs}, Linear Algebra Appl., 588 (2020) 210-223.



\bibitem{Aguilar1}
C.O. Aguilar, J. Lee, E. Piato, and B. J. Schweitzer, \textit{Spectral characterizations of anti-regular graphs}, Linear Algebra Appl., 557 (2018) 84-104.


\bibitem{Ala}
 A. Alazemi, M. Andeli\'c, S. K. Simi\'c, \textit{Eigenvalue location for chain graphs}, Linear Algebra Appl., 505 (2016) 194--210.
 
\bibitem{Tura 5}
L. E. Allem, E. R. Oliveira, F. C. Tura, \textit{Generating I-eigenvalue free threshold graphs}, Electron. J. Combin., 30 (2), (2023) 153-167.

\bibitem{Tura 3}
L. E. Allem, F. C. Tura, \textit{Multiplicity of eigenvalues of cographs}, Discrete Appl. Math., (2018), DOI: 10.1016/j.dam.2018.02.010

\bibitem{Tura 4}
L. E. Allem, F. C. Tura, \textit{Integral cographs}, Discrete Appl. Math., 283 (2020) 153-167.


\bibitem{Bapat}
R. B. Bapat, \textit{On the adjacency matrix of a threshold graph}, Linear Algebra Appl., 439 (2013) 3008-3015.

\bibitem{Bell}
F. K. Bell, D. Cvetković, P. Rowlinson, S. K. Simić, \textit{Graphs for which least eigenvalue is minimal}, I, Linear Algebra Appl., 429 (2008) 234–241.

\bibitem{Stanic}
T. Biyikoglu, A. K. Simi\'c, Z. Stani\'c, \textit{Some notes on spectra of cographs}, Ars Combin., 100 (2011), 421-434.


\bibitem{Chang}
G. J. Chang, L. H. Huang, H. G. Yeh, \textit{On the rank of a cograph}, Linear Algebra Appl., 429 (2008) 601-605.

\bibitem{CT}
G. Chen, F. C. Tura, \textit{A linear algorithm for obtaining the Laplacian eigenvalues of a cograph}, Special Matrices, 12 (1) (2024) pp. 20240024.


\bibitem{Corneil 1}
D. G. Corneil, H. Lerchs, L. Stewart Burlingham, \textit{Complement reducible graphs}, Discrete Appl. Math., 3(1981) 163-174.

\bibitem{KCD1}
K. C. Das, M. Liu, \textit{Complete split graph determined by its (signless) Laplacian spectrum}, Discrete Appl. Math., 205 (2016) 45-51.

\bibitem{Ghorbani 1}
E. Ghorbani, \textit{Spectral properties of cographs and $P_5$-free graphs}, Linear Multilinear Algebra,  (2019), DOI: 10.1080/03081087.2018.1466865.

\bibitem{godsil}
C. Godsil, G. Royle, \textit{Algebraic Graph Theory}, Springer, New York, (2001).

\bibitem{Han1}
P. Hansen, H. Melot, D. Stevanovic, \textit{Integral complete split graphs}, Univ. Beograd. Publ. Elektrotehn. Fak. Ser. Math., 13  (2002), 89-95.

\bibitem{He}
C. He, S. Qian, H. Shan, B. Wu, \textit{The Signature of Chordal Graphs and Cographs}, Graphs Combin., 37 (2021) 643-650.


\bibitem{Horn}
R. A. Horn, C. R. Johnson, \textit{Matrix Analysis}, Cambridge University Press, 2nd Ed, (1994)

\bibitem{Jacobs 2}
D. P. Jacobs, V. Trevisan, F. Tura, \textit{Eigenvalue location in cographs}, Discrete Appl. Math., 245 (2018) 220-235.


\bibitem{Jung}
H. A. Jung, \textit{On a class of posets and and the corresponding comparability graphs}, J. Comb. Theory Ser. B, 24 (1978) 125-133.


\bibitem{Lerchs}
H. Lerchs, \textit{On cliques and kernels}, Tech. Report, Dept. of Comp. Sci. Univ. of Toronto (1971).

\bibitem{Mandal}
S. Mandal, R. Mehatari, KC Das, \textit{On the spectrum and energy of Seidel matrix for chain graphs}, Preprint (2022), available at arXiv:2205.00310

\bibitem{MMZ}
S. Mandal, R. Mehatari, Z. Stani\'c, \textit{Laplacian eigenvalues and eigenspaces of cographs generated by finite sequence}, Indian J. Pure Appl. Math., (2024), DOI-https://doi.org/10.1007/s13226-024-00572-w

\bibitem{Trevisan 1}
A. Mohammadian, V. Trevisan, \textit{Some spectral properties of cographs}, Discrete Math., 339 (2016) 1261-1264.

\bibitem{Pru} H. Pr\"{u}fer, Neuer Beweis eines Satzes \"{u}ber Permutationen, Arch. Math. Phys., 27 (1918) 742--744.

\bibitem{Royle}
G. F. Royle, \textit{The Rank of a Cograph}, Electron. J. Combin., 10 (2003) N11.

\bibitem{Sander}
T. Sander, \textit{On Certain Eigenspace of Cographs}, Electron. J. Combin., 15 (2008) R140.


\bibitem{Seinsche}
D. Seinsche, \textit{On a property of the class of n-colorable graphs}, J. Comb. Theory Ser. B 16 (1974) 191-193.

\bibitem{Song}
G. Song, G. Su, H. Shi, \textit{A complete characterization of bidegreed split graphs with four distinct signless Laplacian eigenvalues}, Linear Algebra Appl., 629 (2021), 232--245.

\bibitem{Sumner}
D. P. Sumner, \textit{Dacey graphs}, J. Aust. Soc., 18 (1974) 492-502.

\end{thebibliography}
\end{document}